\documentclass[11pt]{amsart}
\usepackage{stmaryrd}
\usepackage{amsmath}
\usepackage{amsfonts}
\usepackage{amsthm}
\usepackage{txfonts}
\usepackage{hyperref}
\usepackage{graphicx,amssymb,amsfonts,amsmath,amscd}
\usepackage{    textcomp}
\usepackage[all,ps,cmtip]{xy}
\usepackage{amscd}
\usepackage{xspace}
\usepackage{mathrsfs}

\newtheorem{thm}{Theorem}[section]
\newtheorem{prop}[thm]{Proposition}
\newtheorem{lemma}[thm]{Lemma}
\newtheorem{cor}[thm]{Corollary}
\newtheorem{conj}[thm]{Conjecture}

\newtheorem{def-prop}[thm]{Definition-Proposition}
\newtheorem{prop-def}[thm]{Proposition-Definition}

\theoremstyle{definition} 
\newtheorem{defi}[thm]{Definition}
\newtheorem{rmk}[thm]{Remark}

\newtheorem{expls}[thm]{Examples}

\newcommand{\C}{{\bf C}}

\newcommand{\Z}{{\bf Z}}
\newcommand{\N}{{\bf N}}
\newcommand{\Q}{{\bf Q}}

\newcommand{\sI}{{\mathscr I}}

\newcommand{\sO}{{\mathscr O}}
\newcommand{\sP}{{\mathscr P}}

\newcommand{\rI}{{\mathcal I}}

\newcommand{\rL}{{\mathcal L}}

\newcommand{\rO}{{\mathcal O}}

\newcommand{\ie}{\textit {i.e.}~}
\newcommand{\cf}{\textit {cf.}~}
\newcommand{\eg}{\textit {e.g.}~}

\newcommand{\loccit}{\textit {loc.cit.}~}

\newcommand{\etc}{\textit {etc.}~}
\newcommand{\resp}{\textit {resp.}~}

\newcommand{\CH}{\mathop{\rm CH}\nolimits} 
\newcommand{\ch}{\mathop{\rm ch}\nolimits} 
\newcommand{\cl}{\mathop{\rm cl}\nolimits} 
\newcommand{\dual}{\mathop{^\vee}\nolimits} 
\newcommand{\Gr}{\mathop{\rm Gr}\nolimits} 

\newcommand{\id}{\mathop{\rm id}\nolimits} 
\renewcommand{\P}{\mathop{\bf P}\nolimits} 
\newcommand{\Pic}{\mathop{\rm Pic}\nolimits} 
\newcommand{\pr}{\mathop{\rm pr}\nolimits} 
\newcommand{\Proj}{\mathop{\rm Proj}\nolimits}

\newcommand{\Spec}{\mathop{\rm Spec}\nolimits}
\newcommand{\Sym}{\mathop{\rm Sym}\nolimits} 
\newcommand{\td}{\mathop{\rm td}\nolimits} 

\newcommand{\miu}{\mathop{\bf \mu} \nolimits} 

\renewcommand{\bar}{\overline}
\newcommand{\inj}{\hookrightarrow}
\newcommand{\surj}{\twoheadrightarrow}

\newcommand{\lra}{\xrightarrow}

\newcommand{\cart}{\ar@{}[dr]|\square} 
\renewcommand{\dual}{^{\vee}} 
\newcommand{\isom}{\simeq} 

\renewcommand{\tilde}{\widetilde}

\newcommand{\deff}{\mathrel{:=}}
\newcommand{\m }{\mathop{\mathbf{m}}\nolimits}
\newcommand{\AJ }{\mathop{\rm {AJ}}\nolimits}

\begin{document}

\title[Beauville-Voisin conjecture for generalized Kummer varieties]{Beauville-Voisin Conjecture for generalized Kummer
varieties}
\author{Lie Fu}
\address{D\'epartement de Math\'ematiques et Applications, \'Ecole Normale Sup\'erieure, 45 Rue d'Ulm, 75230 Paris Cedex 05, France}
\email{lie.fu@ens.fr}

\begin{abstract}
Inspired by their results on the Chow rings of projective K3 surfaces, Beauville and Voisin made the following conjecture: given a projective
hyperk\"ahler manifold, for any algebraic cycle which is a
polynomial with rational coefficients of Chern classes of the tangent
bundle and line bundles, it is rationally equivalent to zero if and
only if it is numerically equivalent to zero. In this paper, we prove the Beauville-Voisin conjecture for generalized Kummer varieties.

\end{abstract}
\maketitle

%
%
%
%
%
%
%

\section{Introduction}\label{sect:intro}

In \cite{MR2047674}, Beauville and Voisin observe the following property of the Chow rings of projective K3 surfaces.
\begin{thm}[Beauville-Voisin]\label{thm:K3}
  Let $S$ be a projective K3 surface. Then\\
  $(i)$ There is a well defined 0-cycle $o\in \CH_0(S)$, which is
  represented by any point on any rational curve on $S$. It is
  called \emph{the canonical cycle}.\\
  $(ii)$ For any two divisors $D$, $D'$, the intersection product $D\cdot
  D'$ is proportional to the canonical cycle $o$ in $\CH_0(S)$.\\
  $(iii)$ $c_2(T_S)=24 o \in \CH_0(S)$.\\
  In particular, for any algebraic cycle which is a polynomial on Chern
  classes of the tangent bundle $T_S$ and of line bundles on $S$, it is rationally equivalent to zero if and only if it is numerically equivalent to zero.
\end{thm}
As is pointed out in their paper, the above result is surprising because $\CH_0(S)$ is very huge (`infinite dimensional' in the sense of Mumford \cite{MR0249428}, \cf\cite[Chapter 10]{MR1997577}). In a subsequent paper \cite{MR2187148}, Beauville proposed a conjectural explanation for Theorem \ref{thm:K3} to put it into a larger picture. To explain his idea, let us firstly recall the following notion generalizing K3 surfaces to higher dimensions. See for example \cite{MR730926}, \cite{MR1664696}, or \cite{MR1963559} for a more detailed treatment.

\begin{defi}[\cf\cite{MR730926}] A smooth projective complex variety $X$ is called
\emph{hyperk\"ahler} or \emph{irreducible holomorphic symplectic},
if it is simply connected and $H^{2,0}(X)$ is 1-dimensional and
generated by a holomorphic 2-form which is non-degenerate at each
point of $X$. In particular, a hyperk\"ahler variety has trivial canoncial bundle.
\end{defi}

\begin{expls}\label{examples} Let us give some basic examples of projective
hyperk\"ahler manifolds:
\begin{itemize}
  \item (Beauville \cite{MR730926}) Let $S$ be a projective K3 surface and $n\in \N$, then $S^{[n]}$, which is the Hilbert
  scheme of subschemes of dimension 0 and length $n$, is
  hyperk\"ahler of dimension $2n$.
  \item (Beauville \cite{MR730926}) Let $A$ be an abelian surface and $n\in \N$. Let $s: A^{[n+1]}\to
  A$ be the natural morphism defined by the composition of the Hilbert-Chow morphism $A^{[n+1]}\to A^{(n+1)}$ and the summation $A^{(n+1)}\to A$ using the group law of $A$. It is clear that $s$ is an isotrivial fibration. Then a fibre $K_n\deff s^{-1}\left( O_A\right)$ is hyperk\"ahler
  of dimension $2n$, called \emph{generalized Kummer variety}. The name is justified by the fact that $K_1$ is exactly the Kummer K3 surface associated to $A$.

  \item (Beauville-Donagi \cite{MR818549}) Let $X\subset \P^5$ be a smooth cubic fourfold, then its
  \emph{Fano variety of lines} $F(X)\deff \left\{l\in \Gr\left(\P^1, \P^5\right)~|~ l\subset
  X\right\}$ is hyperk\"ahler of dimension 4.
\end{itemize}
\end{expls}

As an attempt to understand Theorem \ref{thm:K3} in a broader framework, Beauville gives the point of view in \cite{MR2187148} that we can regard this result as a `splitting property' of the conjectural Bloch-Beilinson-Murre filtration on Chow groups (see \cite{MR2115000}, \cite{MR1265533}) for certain varieties with trivial canonical bundle.
He suggests to verify the following down-to-earth consequence of this conjectural splitting of the conjectural filtration on Chow groups of hyperk\"ahler varieties. As a first evidence, the special cases when $X=S^{[2]}$ or $S^{[3]}$ for a projective K3 surface $S$ are verified in his paper \loccit

\begin{conj}[Beauville]\label{conj:Beauville1}
Let $X$ be a projective hyperk\"ahler manifold, and $z\in \CH(X)_\Q$ be
a polynomial with $\Q$-coefficients of the first Chern classes of line
bundles on $X$. Then $z$ is
homologically trivial if and only if $z$ is (rationally equivalent
to) zero.
\end{conj}

Voisin pursues the work of Beauville and makes in \cite{MR2435839} the following stronger version of Conjecture \ref{conj:Beauville1}, by involving also the Chern classes of the tangent bundle:
\begin{conj}[Beauville-Voisin]\label{conj:BeauvilleVoisin}
Let $X$ be a projective hyperk\"ahler manifold, and $z\in \CH(X)_\Q$ be
a polynomial with $\Q$-coefficients of the first Chern classes of line
bundles on $X$ and the Chern classes of the tangent bundle of $X$. Then $z$ is
numerically trivial if and only if $z$ is (rationally equivalent
to) zero.
\end{conj}
Here we replaced `homologically trivial' in the original statement in Voisin's paper \cite{MR2435839} by `numerically trivial'. But according to the standard conjecture \cite{MR1265519}, the homological equivalence and the numerical equivalence are expected to coincide. We prefer to state the Beauville-Voisin conjecture in the above slightly stronger form since our proof for generalized Kummer varieties also works in this generality.

In \cite{MR2435839}, Voisin proves Conjecture \ref{conj:BeauvilleVoisin} for the Fano varieties of lines of cubic fourfolds, and for $S^{[n]}$ if $S$ is a projective K3 surface and $n\leq 2b_{2,tr}+4$, where $b_{2,tr}$ is the second Betti number of $S$ minus its Picard number. We remark that here we indeed can replace the homological equivalence by the numerical equivalence since the standard conjecture in these two cases has been verified by Charles and Markman \cite{MR3040747}.

The main result of this paper is to prove the Beauville-Voisin conjecture \ref{conj:BeauvilleVoisin} for generalized
Kummer varieties.

\begin{thm}\label{thm:main2}
  Let $A$ be an abelian surface, $n\geq 1$ be a natural number.
  Denote by $K_n$ the generalized Kummer variety associated to $A$ (\cf Examples \ref{examples}). Consider any algebraic cycle $z\in \CH(K_n)_\Q$ which is a polynomial with rational coefficients of the first Chern classes of line bundles on $K_n$ and the Chern classes of the tangent bundle of $K_n$, then $z$ is numerically trivial if and only if $z$ is (rationally equivalent to) zero.
\end{thm}

There are two key ingredients in the proof of the above theorem: on the one hand, as in \cite{MR2435839}, the result of De Cataldo-Migliorini \cite{MR1919155} recalled in Section \ref{sect:DeCM} relates the Chow groups of $A^{[n]}$ to the Chow groups of various products of $A$. On the other hand, a recent result on algebraic cycles on abelian varieties due to Moonen \cite{MoonenNew} and O'Sullivan \cite{MR2795752}, which is explained in Section \ref{sect:MO}, allows us to upgrade a relation modulo numerical equivalence to a relation modulo rational equivalence.

\noindent{\textbf{Convention: }} Throughout this paper, we work over
the field of complex numbers. All Chow groups are with rational
coefficients $\CH\deff\CH\otimes \Q$. 
If $A$ is an abelian variety, we denote by $O_A$ its origin and $\Pic^s(A)$ its group of symmetric line bundles. For any smooth projective surface $S$, we denote by $S^{[n]}$ the Hilbert scheme of subschemes of length $n$, which is a $2n$-dimensional smooth projective variety by \cite{MR0237496}. Finally, for an algebraic variety $X$, the  \emph{big diagonal $\Delta_{ij}$} in a self-product $X^n$ is the subvariety $\left\{(x_1, \cdots, x_n)\in X^n~|~~ x_i=x_j\right\}$.

\section{De Cataldo-Migliorini's result}\label{sect:DeCM}
As mentioned above, a crucial ingredient for the proof of Theorem
\ref{thm:main2} will be the following result due to De Cataldo and
Migliorini. We state their result in the form adapted to our
purpose.

Let $S$ be a projective surface, $n\in \N_+$ and $\sP(n)$ be the set
of partitions of $n$. For any such partition $\miu=(\miu_1, \cdots,
\miu_l)$, we denote by $l_{\miu}:=l$ its length. Define
$S^{\miu}:=S^{l_{\miu}}=\underbrace{S\times\cdots\times
S}_{l_{\miu}}$, and also a natural morphism from it to the symmetric
product:
\begin{eqnarray*}
   S^{\miu}  & \to & S^{(n)}\\
   \left(x_1,\cdots, x_l\right) & \mapsto& \miu_1x_1+\cdots+\miu_lx_l.
\end{eqnarray*}
Now define $E_{\miu}\deff \left(S^{[n]}\times_{S^{(n)}}S^{\miu}\right)_{red}$
to be the reduced incidence variety inside $S^{[n]}\times S^{\miu}$. Then
$E_{\miu}$ can be viewed as a correspondence from $S^{[n]}$ to
$S^{\miu}$, and we will write ${}^tE_{\miu}$ for the
\emph{transpose} correspondence, namely the correspondence from
$S^{\miu}$ to $S^{[n]}$ defined by the same subvariety $E_{\miu}$ in
the product. Let $\miu=(\miu_1, \cdots,
\miu_l)=1^{a_1}2^{a_2}\cdots n^{a_n}$ be a partition of $n$, we define $m_{\miu}:=(-1)^{n-l}\prod_{j=1}^{l}\miu_j$ and $c_{\miu}:= \frac{1}{m_{\miu}}\frac{1}{a_1!\cdots a_n!}$.

\begin{thm}[De Cataldo-Migliorini \cite{MR1919155}]\label{thm:CataldoMigliorini}
Let $S$ be a projective surface, $n\in \N_+$. For each $\miu\in
\sP(n)$, let $E_{\miu}$ and ${}^tE_{\miu}$ be the correspondences
defined above. Then the sum of the compositions
$$\sum_{\miu\in \sP(n)}c_{\miu}~{}^t\!E_{\miu}\circ E_{\miu}=\Delta_{S^{[n]}}$$
is the identity correspondence of $S^{[n]}$, modulo rational
equivalence. In particular,  $$\sum_{\miu\in \sP(n)}c_{\miu}~E_{\miu}^*\circ E_{\miu *}=\id_{\CH(S^{[n]})}\colon \CH(S^{[n]}) \to \CH(S^{[n]}).$$
\end{thm}


Return to the case where $S=A$ is an abelian surface. We view
$A^{[n+1]}$ as a variety over $A$ by the natural summation morphism
$s\colon A^{[n+1]}\to A$.
 Similarly, for each $\miu\in\sP(n+1)$ of
length $l$, $A^{\miu}$ also admits a natural morphism to $A$,
namely, the \emph{weighted sum}:
\begin{eqnarray*}
s_{\miu}\colon A^{\miu} &\to & A\\
\left(x_1, \cdots, x_{l}\right)&\mapsto& \miu_1x_1+\cdots+\miu_lx_l.
\end{eqnarray*}

By definition, the
correspondences $E_{\miu}$, ${}^tE_{\miu}$ are compatible with
morphisms $s$ and $s_{\miu}$ to $A$, \ie the following diagram
commutes:
\begin{displaymath}
  \xymatrix{
   &E_{\miu}\ar[dl] \ar[dr]\ar[dd]^{\pi_{\miu}}&\\
   A^{[n+1]} \ar[dr]_{s} & &A^{\miu}\ar[dl]^{s_{\miu}}\\
   &A&.
  }
\end{displaymath}

We point out that the three morphisms to $A$ are all isotrivial fibrations: they become products after the base change $A\lra{\cdot n+1} A$ given by multiplication by $n+1$. Now let us take their fibres over the origin of $A$, or equivalently,  apply the base change $i: \Spec(\C)=O_A\inj A$ to the above commutative diagram, we obtain the following correspondence, where $K_n:=s^{-1}\left(O_A\right)$ is the generalized Kummer variety, $B_{\miu}$ is the possibly non-connected abelian variety $B_{\miu}\deff\ker\left(s_{\miu}\colon A^{\miu}\to A\right)$, and $\Gamma_{\miu}:=\pi_{\miu}^{-1}\left(O_A\right)$.

\begin{displaymath}
  \xymatrix{
   &\Gamma_{\miu}\ar[dl] \ar[dr]&\\
   K_n \ar[dr] & &B_{\miu}\ar[dl]\\
   &O_A=\Spec(\C)&.
  }
\end{displaymath}

In the sequel, we sometimes view $E_{\miu}$ simply as an algebraic cycle in $\CH\left(A^{[n+1]}\times A^{\miu}\right)$ and also by definition $\Gamma_{\miu}=i^!\left(E_{\miu}\right)\in \CH\left(K_n\times B^{\miu}\right)$, where $i^!$ is the \emph{refined Gysin map} defined in \cite[Chapter 6]{MR1644323}. We need the following standard fact in intersection theory.

\begin{lemma}\label{lemma: intersection}
For any $\gamma\in \CH\left(A^{[n+1]}\right)$, we have $$\Gamma_{\miu *}\left(\gamma|_{K_n}\right)=\left(E_{\miu *}(\gamma)\right)|_{B_{\miu}} ~~\text{in}~~\CH\left(B_{\miu}\right).$$
Similarly, for any $\beta\in \CH\left(A^{\miu}\right)$, we have $$\Gamma_{\miu}^*\left(\beta|_{B_{\miu}}\right)=\left(E_{\miu}^*(\beta)\right)|_{K_n} ~~\text{in}~~\CH\left(K_{n}\right).$$
\end{lemma}
\begin{proof}
All squares are cartesian in the following commutative diagram. 
\begin{displaymath}
\xymatrix{
K_n\times B_{\miu}\ar[r]^{q'} \ar[d]_{p'} \ar@{^{(}->}[drr] & B_{\miu} \ar[d] \ar@{^{(}->}[drr] & & \\
K_n \ar[r] \ar@{^{(}->}[drr] & O_A \ar@{^{(}->}[drr]|(0.3){i} & A^{[n+1]}\times A^{\miu}\ar[d]^(0.35){p} \ar[r]_(0.6){q}& A^{\miu}\ar[d]^{s_{\miu}}\\
&&A^{[n+1]}\ar[r]_{s} & A.\\
}
\end{displaymath}
Now for any $\gamma\in \CH\left(A^{[n+1]}\right)$, we have
\begin{eqnarray*}
\Gamma_{\miu *}\left(\gamma|_{K_n}\right)&=& \Gamma_{\miu *}\left( i^!(\gamma)\right) ~~~~~\text{(by \cite[Theorem 6.2(c)]{MR1644323}, as $s$ is isotrivial)}\\
&=& q'_*\left(p'^*\left( i^!(\gamma)\right) \cdot i^!\left(E_{\miu}\right)\right)\\
&=& q'_*\left(i^!\left(p^*(\gamma)\right) \cdot i^!\left(E_{\miu}\right)\right) ~~~~~\text{(by \cite[Theorem 6.2(b)]{MR1644323})}\\
&=& q'_*\left(i^!\left(p^*(\gamma)\cdot E_{\miu}\right)\right)\\
&=& i^!\left(q_*\left(p^*(\gamma)\cdot E_{\miu}\right)\right) ~~~~~\text{(by \cite[Theorem 6.2(a)]{MR1644323})}\\
&=& i^!\left(E_{\miu *}(\gamma)\right)\\
&=& \left(E_{\miu *}(\gamma)\right)|_{B_{\miu}}  ~~~~~\text{(by \cite[Theorem 6.2(c)]{MR1644323}, as $s_{\miu}$ is isotrivial)}\\
\end{eqnarray*}
The proof of the second equality is completely analogous.
\end{proof}

Theorem \ref{thm:CataldoMigliorini} together with Lemma \ref{lemma: intersection} implies the following

\begin{cor}\label{cor:CataldoMigliorini}
For each $\miu\in \sP(n+1)$, let $\Gamma_{\miu}$ be the
correspondences between $K_n$ and $B_{\miu}$ defined above. Then  for any $\gamma\in \CH\left(A^{[n+1]}\right)$, we have $$\sum_{\miu\in \sP(n+1)}c_{\miu}\Gamma_{\miu}^*\circ \Gamma_{\miu *}(\gamma|_{K_n})=\gamma|_{K_n} ~~\text{in} ~\CH(K_n),$$
where for a partition $\miu=(\miu_1, \cdots,
\miu_l)=1^{a_1}2^{a_2}\cdots (n+1)^{a_{n+1}}\in \sP(n+1)$, the constant $c_{\miu}$ is defined as $\frac{1}{(-1)^{n+1-l}\prod_{j=1}^{l}\miu_j}\cdot\frac{1}{a_1!\cdots a_{n+1}!}$.
\end{cor}

For later use, we now describe $B_{\miu}$. Let
$d:=\gcd\left(\miu_1,\cdots,\miu_l\right)$, then $B_{\miu}$ has $d^4$
isomorphic connected components. We denote by $B^0_{\miu}$ the
identity component, which is a connected abelian variety; and the other components are its torsion
translations. More precisely, define the \emph{weighted sum} homomorphism
\begin{eqnarray*}
s_{\miu}\colon \Z^{\oplus l_{\miu}} &\to & \Z\\
\left(m_1, \cdots, m_{l}\right)&\mapsto& \miu_1m_1+\cdots+\miu_lm_l,
\end{eqnarray*}
whose image is clearly $d\Z$. Let $U$ be the kernel of $s_{\miu}$, which is a free abelian group of rank $l_{\miu}-1$. Define the \emph{reduced weighted sum}
\begin{eqnarray*}
\bar s_{\miu}\colon \Z^{\oplus l_{\miu}} &\to & \Z\\
\left(m_1, \cdots, m_{l}\right)&\mapsto& \frac{\miu_1}{d}m_1+\cdots+\frac{\miu_l}{d}m_l.
\end{eqnarray*}
Then we have a short exact sequence of free abelian groups
\begin{equation}\label{eqn:ses1}
  0\to U\to \Z^{\oplus l_{\miu}} \lra{\bar s_{\miu}}  \Z\to 0.
\end{equation}
By tensoring with $A$, we obtain a short exact sequence of abelian varieties
\begin{equation}\label{eqn:ses2}
  0\to B^0_{\miu}\to A^{\miu} \lra{\bar s_{\miu}}  A\to 0.
\end{equation}

Since the short exact sequence (\ref{eqn:ses1}) splits, so does the short exact sequence (\ref{eqn:ses2}):
$B^0_{\miu}$ is a direct summand of $A^{\miu}$, thus we can choose a
projection $p_{\miu}\colon A^{\miu}\surj B^0_{\miu}$ such that
$p_{\miu}\circ i_{\miu}=\id_{B^0_{\miu}}$, where $i_{\miu}\colon
B^0_{\miu}\inj A^{\miu}$ is the natural inclusion.

Denoting $A[d]$ for the set of $d$-torsion points of $A$, we have
$$B_{\miu}=\bigsqcup_{t\in A[d]}B^t_{\miu},$$
where $B^t_{\miu}\deff \left\{(x_1, \cdots, x_l)\in A^{\miu}~|~\sum_{i=1}^{l}\frac{\miu_i}{d}x_i=t\in A\right\}$.

Now we specify the way that we view $B^t_{\miu}$ as a torsion translation of $B^0_{\miu}$. Since $d$ is the greatest common divisor of ${\miu}_1, \cdots, {\miu}_l$, it divides $n+1$. We choose $t'\in A[n+1]$ such that $\frac{n+1}{d}\cdot t'=t$ in $A$. Then the torsion translation on $A$ by $t'$ will induce some `torsion translation automorphism' $\tau_{t'} \deff (t', \cdots, t')$ on $A^{[n+1]}$
\begin{eqnarray*}
  \tau_{t'}: A^{[n+1]} & \to & A^{[n+1]}\\
  z & \mapsto & z+t',
\end{eqnarray*}
(\eg when $z$ is a reduced subscheme of length $n+1$ given by  $(x_1, \cdots, x_{n+1})$ with $x_j$'s pairwise distinct, it is mapped to $z+t':=(x_1+t', \cdots, x_{n+1}+t')$);
as well as on $A^{\miu}$
\begin{eqnarray*}
  \tau_{t'}: A^{\miu} &\to& A^{\miu}\\
  (x_1, \cdots, x_l)&\mapsto& (x_1+t', \cdots, x_l+t').
\end{eqnarray*}
These actions are compatible: we have the following commutative diagram with actions:
\begin{displaymath}
  \xymatrix{
   &\tau_{t'} \circlearrowright E_{\miu}\ar[dl] \ar[dr]&\\
   \tau_{t'} \circlearrowright A^{[n+1]} \ar[dr]_{s} & &A^{\miu} \circlearrowleft\tau_{t'}\ar[dl]^{s_{\miu}}\\
   &\id \circlearrowright A&
  }
\end{displaymath}
Moreover, the action of $\tau_{t'}$ on $A^{\miu}$  translates $B^0_{\miu}$ isomorphically to $B^t_{\miu}$.

\section{Result of Moonen and O'Sullivan}\label{sect:MO}
In this section,  $A$ is an abelian variety of dimension $g$. For any $m\in \Z$, let $\m$ be the
endomorphism of $A$ defined by the multiplication by $m$.  To motivate the result of Moonen and O'Sullivan,  let us firstly recall the Beauville conjectures for algebraic cycles on abelian varieties. In \cite{MR726428} and \cite{MR826463},  Beauville investigates the Fourier transformation between the Chow rings of $A$ and its dual abelian variety $\hat A$ and establishes the following:

\begin{thm}[Beauville decomposition]\label{thm:Beauville}
Let $A$ be a $g$-dimensional abelian variety.\\
 $(i)$ For any $0\leq
i\leq g$, there exists a direct-sum decomposition
$$\CH^i(A)=\bigoplus_{s=i-g}^i\CH^i_{(s)}(A),$$ where $\CH^i_{(s)}(A)\deff\left\{z\in \CH^i(A)~|~\m^*z=m^{2i-s}z , \forall m\in \Z
\right\}$.\\
$(ii)$ This decomposition is functorial: Let $B$ be another abelian
variety of dimension $(g+c)$ and $f\colon A\to B$ be a homomorphism
of abelian varieties. Then for any $i$,
$$f^*\left(\CH^i_{(s)}(B)\right)\subset \CH^i_{(s)}(A);$$
$$f_*\left(\CH^i_{(s)}(A)\right)\subset \CH^{i+c}_{(s)}(B).$$
$(iii)$ The intersection product respects the grading: $\CH^i_{(s)}(A)\cdot \CH^j_{(t)}(A)\subset \CH^{i+j}_{(s+t)}(A)$.
\end{thm}

In the spirit of Bloch-Beilinson-Murre conjecture (\cf \cite{MR2115000}, \cite{MR1265533}), Beauville makes in \cite{MR826463} the
following conjectures, which roughly say that $F^j\CH^i(A)\deff\oplus_{s\geq
j}\CH^i_{(s)}(A)$ should give the desired conjectural Bloch-Beilinson-Murre filtration.

\begin{conj}[Beauville conjectures]\label{conj:Beauville}
$(i)$ For any $i$ and any $s<0$, $\CH^i_{(s)}(A)=0$;\\
$(ii)$ For any $i$, the restriction of the cycle class map
$\cl\colon \CH_{(0)}^i(A)\to H^{2i}(A,\Q)$ is injective;\\
$(iii)$ For any $i$, the restriction of the Abel-Jacobi map
$\AJ\colon \CH_{(1)}^i(A)\to J^{2i-1}(A)_\Q$ is injective.
\end{conj}

Obviously, the Beauville conjectures hold for divisors, \ie $\CH^1(A)=\CH^1_{(0)}(A)\oplus \CH^1_{(1)}(A)$ where $$\CH^1_{(0)}(A)=\Pic^s(A)_\Q\isom NS(A)_\Q;$$
$$\CH^1_{(1)}(A)=\Pic^0(A)_\Q=\hat A\otimes_\Z\Q.$$

In particular, the $\Q$-subalgebra of $\CH^*(A)$ generated by symmetric line bundles on $A$ is contained in $\CH^{*}_{(0)}(A)$ (by Theorem \ref{thm:Beauville}(iii)). As a special case of Beauville's Conjecture \ref{conj:Beauville}(ii), Voisin raised the natural question whether the cycle class map $\cl$ is injective on this subalgebra. Recently,  Moonen \cite[Corollary 8.4]{MoonenNew} and O'Sullivan \cite[Theorem Page 2-3]{MR2795752} have given a positive answer to Voisin's question:

\begin{thm}[Moonen, O'Sullivan]\label{thm:MO}
Let $A$ be an abelian variety. Let $P\in \CH^*(A)$ be a polynomial with rational coefficients in the first Chern classes of symmetric line bundles on $A$, then $P$ is numerically equivalent to zero if and only if $P$ is (rationally equivalent to) zero.
\end{thm}

\begin{rmk}
  The above result is implicit in O'Sullivan's paper \cite{MR2795752}. In fact, he constructs the so-called \emph{symmetrically distinguished cycles} $\CH^*(A)_{sd}$, which is a $\Q$-subalgebra of $\CH^*(A)$ containing the first Chern classes of symmetric line bundles and mapping isomorphically by the numerical cycle class map to $\bar{\CH^*}(A)$, the $\Q$-algebra of cycles modulo the numerical equivalence.
\end{rmk}

\section{Proof of Theorem \ref{thm:main2}}

Let us prove the main result. To fix the notation, we recall the following description of line
bundles on $K_n$ (see \cite[Proposition 8]{MR730926}). Let
$\epsilon: A^{[n+1]}\to A^{(n+1)}$ be the Hilbert-Chow morphism,
which is a resolution of singularities (\cite{MR0237496}).

\begin{prop}[Beauville]\label{prop:NS}
We have an injective homomorphism
\begin{eqnarray*}
j: NS(A)_\Q&\inj& NS(K_n)_\Q\\
c_{1}^{top}(L) & \mapsto & \tilde L|_{K_n}
\end{eqnarray*}
such that $$\Pic(K_n)_\Q=NS(K_n)_\Q=j\left(NS(A)_\Q\right)\oplus \Q
\cdot \delta|_{K_n},$$ where $\delta$ is the exceptional divisor
of $A^{[n+1]}$.
\end{prop}
Here for a line bundle $L$
on $A$, the $\mathfrak{S}_{n+1}$-invariant line bundle $L\boxtimes
\cdots\boxtimes L$ on $A\times\cdots \times A$ descends to a line
bundle $L'$ on the symmetric product $A^{(n+1)}$ and we define
$\tilde L:=\epsilon^*(L')$.

\begin{rmk}\label{rmk:numericalL}
 As the notation in this proposition indicates,  modifying the line bundle $L$ on $A$ inside its numerical equivalence class will not change the resulting line bundle $j(L)=\tilde L|_{K_n}\in NS (K_n)_\Q=\Pic (K_n)_\Q$.
\end{rmk}

We hence obtain the following
\begin{lemma}\label{lemma:restriction}
Given any  polynomial $z\in \CH(K_n)$ in the Chern classes of $T_{K_n}$ and the first Chern classes of line bundles on $K_n$, as in the main theorem, then:\\
(i) There exists $\gamma\in \CH(A^{[n+1]})$ which is a polynomial of algebraic cycles
of one of the three forms: $c_1\left(\tilde L\right)$ for some symmetric line bundle $L\in
\Pic^s(A)_\Q$, $\delta$, and $c_j\left(T_{A^{[n+1]}}\right)$ for some $j\in
\N$, such that $$\gamma|_{K_n}=z ~~\text{   in  }~~ \CH(K_n).$$
(ii) Moreover, for such $\gamma$, the automorphism $\tau_{t'}$ of $A^{[n+1]}$ constructed at the end of Section \ref{sect:DeCM} satisfies $$\left(\tau_{t'*}(\gamma)\right)|_{K_n}=\gamma|_{K_n}=z ~~\text{   in  }~~ \CH(K_n).$$
\end{lemma}
\begin{proof}
(i) Note that $c_j\left(T_{K_n}\right)=c_j\left(T_{A^{[n+1]}}\right)|_{K_n}$, since $T_A$ is trivial. Part $(i)$ thus follows from Proposition  \ref{prop:NS} because $c^{top}_1\colon \Pic^s(A)_\Q\lra{\isom} NS(A)_\Q$ is an isomorphism (see \cite[Page 649]{MR826463}).\\
  (ii) It is clear that $\tau_{t'*}(\delta)=\delta$ and $\tau_{t'*}(T_{A^{[n+1]}})=T_{A^{[n+1]}}$. On the other hand, the pushforward of $L$ by a torsion translation on $A$ has the same numerical class as $L$ and hence by Remark \ref{rmk:numericalL}, $\left(\tau_{t'*}(\tilde L)\right)|_{K_n}=\tilde L|_{K_n}$ as line bundles. Therefore modifying $\gamma$ by the automorphism $\tau_{t'}$ does not change its restriction to $K_n$, although it might change the cycle $\gamma$ itself.
\end{proof}

Let us start the proof of Theorem \ref{thm:main2}. We will use
$\equiv$ to denote the numerical equivalence. Given $z\in \CH(K_n)$
a polynomial of the Chern classes of $T_{K_n}$ and line bundles on
$K_n$ as in the main theorem \ref{thm:main2}. By Lemma \ref{lemma:restriction}(i), we
can write $z=\gamma|_{K_n}$ for $\gamma\in \CH(A^{[n+1]})$ a
polynomial of $c_1\left(\tilde L\right)$ for some $L\in \Pic^s(A)_\Q$,
$\delta$, and $c_j\left(T_{A^{[n+1]}}\right)$ for some $j\in \N$.

Assuming $z\equiv 0$, we want to prove that $z=0$.  Adopting the previous notation, then for any $\miu\in\sP(n+1)$ we have by Lemma \ref{lemma: intersection}
\begin{equation*}
 \left(E_{\miu*}(\gamma)\right)|_{B_{\miu}}=\Gamma_{\miu*}\left(\gamma|_{K_n}\right)=\Gamma_{\miu*}(z)\equiv 0.
\end{equation*}

Define $\beta\deff E_{\miu*}(\gamma)\in \CH\left(A^{\miu}\right)$, the above equality says that $\beta|_{B_{\miu}}\equiv 0$, in particular,
\begin{equation}\label{eqn:1}
\beta|_{B^0_{\miu}}\equiv 0.
\end{equation}

To describe $\beta$, we need
the following proposition, which is the analogue of
the corresponding result \cite[Proposition 2.4]{MR2435839} due to
Voisin, and we will give its proof in the next section.

\begin{prop}\label{prop:Voisinhard} For $\gamma\in \CH(A^{[n+1]})$ as above (\ie a
polynomial of cycles of the forms: $c_1\left(\tilde L\right)$ for some $L\in \Pic^s(A)_\Q$,
$\delta$, and $c_j\left(T_{A^{[n+1]}}\right)$ for some $j\in \N$), the
algebraic cycle $\beta=E_{\miu*}(\gamma)\in \CH\left(A^{\miu}\right)$ is a
polynomial with rational coefficients in cycles of the two forms:
\begin{itemize}
\item $\pr_i^*\left(L\right)$ for some symmetric line bundle $L$ on $A$
and $1\leq i\leq l_{\miu}$;
\item big diagonal $\Delta_{ij}$ of $A^{\miu}=A^{l_{\miu}}$ for
$1\leq i\neq j\leq l_{\miu}$.
\end{itemize}
\end{prop}
See the next section for its proof.

\begin{cor}\label{cor:Voisinhard}
 With the same notation, $\beta$ is a polynomial with rational coefficients in algebraic cycles of the form
$\phi^*\left(L\right)$, for some homomorphism of abelian varieties $\phi: A^{\miu}\to A$ and some $L\in \Pic^s(A)_\Q$.
\end{cor}
\begin{proof}
 It is enough to remark that the big diagonal $\Delta_{ij}$ is nothing but the pull-back of $O_A\in \CH_0(A)$ via the homomorphism
 \begin{eqnarray*}
  A^{\miu} &\to& A\\
  (x_1,\cdots,x_{l_{\miu}}) &\mapsto& x_i-x_j
 \end{eqnarray*}
and  $O_A\in \CH_0(A)$ is proportional to
$\theta^2$ for some symmetric polarization $\theta\in \Pic^{s}(A)_\Q$, \cf
\cite[Page 249, Corollaire 2]{MR726428}.
\end{proof}

Let us continue the proof of Theorem \ref{thm:main2}. Let
$B^0_{\miu}$ be the identity component of $B_{\miu}$,
$i_{\miu}\colon B^0_{\miu}\inj A^{\miu}$ and $p_{\miu}\colon
A^{\miu}\surj B^0_{\miu}$ be the inclusion and the splitting
constructed in Section \ref{sect:DeCM}. By assumption, we have equation (\ref{eqn:1}): $\beta|_{B_{\miu}}\equiv
0$, therefore $i_{\miu}^*(\beta)=\beta|_{B^0_{\miu}}\equiv 0$, hence
$$p_{\miu}^*\left(i_{\miu}^*(\beta)\right)\equiv 0.$$

On the other hand, since $i_{\miu}\circ
p_{\miu}\colon A^{\miu}\to A^{\miu}$ is an endomorphism of
$A^{\miu}$, Corollary \ref{cor:Voisinhard} implies that the
numerically trivial cycle $p_{\miu}^*\left(i_{\miu}^*(\beta)\right)$ is also a polynomial of cycles of the form $\phi^*\left(L\right)$, for some homomorphism of abelian varieties $\phi: A^{\miu}\to A$ and some $L\in \Pic^s(A)_\Q$. As a result, $p_{\miu}^*\left(i_{\miu}^*(\beta)\right)$ is in the
subalgebra of $\CH\left(A^{\miu}\right)$ generated by the first Chern classes of symmetric line bundles of $A^{\miu}$.

Therefore by the result of Moonen and O'Sullivan (Theorem \ref{thm:MO}),
$p_{\miu}^*\left(i_{\miu}^*(\beta)\right)\equiv 0$ implies
$p_{\miu}^*\left(i_{\miu}^*(\beta)\right)=0$. As a result,
\begin{equation}\label{eqn:2}
  \beta|_{B^0_{\miu}}=i_{\miu}^*\left(p_{\miu}^*\left(i_{\miu}^*(\beta)\right)\right)=0.
\end{equation}
Recall that $d=\gcd(\miu_1, \cdots, \miu_l)$ and  for any $d$-torsion point $t$ of $A$, the automorphism $\tau_{t'}$ constructed at the end of Section \ref{sect:DeCM} translates $B^0_{\miu}$ to $B^t_{\miu}$, therefore (\ref{eqn:2}) implies that
$$\tau_{t'*}(\beta)|_{B^t_{\miu}}=0 \text{  for any  }t\in A[d].$$
However, $\tau_{t'*}(\beta)=\tau_{t'*}\left(E_{\miu*}(\gamma)\right)=E_{\miu*}\left(\tau_{t'*}(\gamma)\right)$ by the compatibility of the actions of $\tau_{t'}$ on $A^{[n+1]}$ and on $A^{\miu}$, as explained in Section \ref{sect:DeCM}.

We thus obtain that for any $t\in A[d]$,
$$\Gamma_{\miu*}\left(z\right)|_{B^t_{\miu}}=\Gamma_{\miu*}\left(\tau_{t'*}\gamma|_{K_n}\right)|_{B^t_{\miu}}
=\left(E_{\miu*}(\tau_{t'*}\gamma)\right)|_{B^t_{\miu}}=\tau_{t'*}(\beta)|_{B^t_{\miu}}=0.$$
Here the first equality comes from Lemma \ref{lemma:restriction}(ii), see also Remark \ref{rmk:numericalL}; the second equality uses Lemma \ref{lemma: intersection}.
Since $B_{\miu}$ is the disjoint union of all $B_{\miu}^t$ for all $t\in A[d]$, we have
$$\Gamma_{\miu*}(z)=0 ~~\text{ for any }~ \miu\in \sP(n+1).$$
Using De Cataldo-Migliorini's result (rather Corollary
\ref{cor:CataldoMigliorini}), we have for $z=\gamma|_{K_n}$ as before,
$$z=\sum_{\miu\in \sP(n+1)}c_{\miu}~\Gamma_{\miu}^*\circ \Gamma_{\miu *}\left(z\right)=0.$$
The proof of Theorem \ref{thm:main2} is complete if one admits
Proposition \ref{prop:Voisinhard}.

\section{The proof of Proposition \ref{prop:Voisinhard}}
The proof of Proposition \ref{prop:Voisinhard} is quite technical but analogous to
that of \cite[Proposition 2.4]{MR2435839}. For the convenience of readers, we give in this section a more or less self-contained proof closely following \cite{MR2435839}, emphasizing the differences from the case in
\cite{MR2435839}. The author thanks Claire Voisin for allowing him to reproduce her arguments.
For simplicity, we switch from $n+1$ to $n$. Let $A$ still be an abelian surface.

There are two natural vector bundles on $A^{[n]}$. The first one is
the tangent bundle $T_n\deff T_{A^{[n]}}$, the second one is the
rank $n$ vector bundle $\rO_n\deff \pr_{1 *}\left(\sO_{U_n}\right)$, where
$U_n\subset A^{[n]}\times A$ is the universal subscheme and
$\pr_1\colon A^{[n]}\times A\to A^{[n]}$ is the first projection. As
$c_1\left(\rO_n\right)=-\frac{1}{2}\delta$, we can generalize Proposition \ref{prop:Voisinhard} by
proving it for any $\gamma$ a polynomial of $c_1\left(\tilde L\right)$ for some $L\in
\Pic^s(A)_\Q$, $c_i\left(\rO_n\right)$ for some $i\in \N$, and $c_j\left(T_n\right)$
for some $j\in \N$.

For any $L\in \Pic^s\left(A\right)$, by the construction of $E_{\miu}\subset
A^{[n]}\times A^{\miu}$, the restriction $\pr_1^*\left(\tilde
L\right)|_{E_{\miu}}$ is the pull back of the line bundle
$L^{\otimes\miu_1}\boxtimes\cdots\boxtimes L^{\otimes\miu_l}$ on
$A^{\miu}$. Hence by projection formula, we only need to prove the
following

\begin{prop}\label{prop:1}
For $\gamma\in \CH(A^{[n]})$ a polynomial with rational coefficients of cycles of the forms:
\begin{itemize}
  \item $c_i\left(\rO_n\right)$ for some $i\in \N$;
  \item $c_j\left(T_n\right)$ for some $j\in \N$,
\end{itemize}
the algebraic cycle $\beta=E_{\miu*}(\gamma)\in \CH\left(A^{\miu}\right)$ is
a polynomial with rational coefficients in the big diagonals
$\Delta_{ij}$ of $A^{\miu}=A^{l_{\miu}}$ for $1\leq i\neq j\leq
l_{\miu}$.
\end{prop}

To show Proposition \ref{prop:1}, we actually prove the more general Proposition \ref{prop:2} below (note that Proposition \ref{prop:1} corresponds to the special case $m=0$), which allows us to do induction on $n$. Let us introduce some notation first: for any $m\in \N$, let $E_{\miu,
m}$ be the correspondence between $A^{[n]}\times A^m$ and
$A^{\miu}\times A^{m}$ defined by $E_{\miu, m}\deff E_{\miu}\times
\Delta_{A^m}$. Let $\rI_n$ be the ideal sheaf of the universal
subscheme $U_n\subset A^{[n]}\times A$. For any $1\leq i\neq j\leq m$, we
denote by $\pr_0\colon A^{[n]}\times A^m\to A^{[n]}$, \resp
$\pr_i\colon A^{[n]}\times A^m\to A$, \resp $\pr_{0i}\colon
A^{[n]}\times A^m\to A^{[n]}\times A$, \resp $\pr_{ij}\colon A^{[n]}\times A^m\to A\times A$ the projection onto  the  factor $A^{[n]}$, \resp the $i$-th factor of $A^m$, \resp the product of the factor $A^{[n]}$ and the $i$-th factor of $A^m$, \resp the product of the $i$-th and $j$-th factors of $A^m$.

\begin{prop}\label{prop:2}
For $\gamma\in \CH(A^{[n]}\times A^m)$ a polynomial with rational coefficients of cycles of the
forms:
\begin{itemize}
  \item $\pr_0^*\left(c_j\left(\rO_n\right)\right)$ for some $j\in \N$;
  \item $\pr_0^*\left(c_j\left(T_n\right)\right)$ for some $j\in \N$;
  \item $\pr_{0i}^*\left(c_j\left(\rI_n\right)\right)$ for some $1\leq i\leq m$ and $j\in \N$;
  \item $\pr_{ij}^*(\Delta_A)$ for some $1\leq i\neq j\leq m$,
\end{itemize}
the algebraic cycle $E_{\miu,m *}(\gamma)\in \CH\left(A^{l_{\miu}+m}\right)$
is a polynomial with rational coefficients in the big diagonals
$\Delta_{ij}$ of $A^{l_{\miu}+m}$, for $1\leq i\neq j\leq
l_{\miu}+m$.
\end{prop}

The main tool to prove this proposition is the so-called \emph{nested Hilbert schemes}, which we briefly recall here (\cf \cite{MR1711344}).
By definition, the nested Hilbert scheme is the incidence variety $$A^{[n-1, n]}:= \left\{(z', z)\in A^{[n-1]}\times A^{[n]}~|~ z'\subset z\right\},$$ where $z'\subset z$ means $z'$ is a closed subscheme of $z$. It admits natural projections to $A^{[n-1]}$ and $A^{[n]}$, and also a natural morphism to $A$ which associates the residue point to such a pair of subschemes $(z'\subset z)$. The situation is summarized by the following diagram:
\begin{equation}\label{diagram:1}
\xymatrix{
A^{[n-1]} & A^{[n-1, n]} \ar[l]_{\phi} \ar[r]^{\psi} \ar[d]^{\rho} & A^{[n]} \\
&A&
}
\end{equation}

We collect here some basic properties of the nested Hilbert scheme (\cf \cite{MR1432198}, \cite{MR2095898}, \cite{MR1711344}): 
\begin{itemize}
\item The nested Hilbert scheme $A^{[n-1, n]}$ is irreducible and smooth  of dimension $2n$ (\cf \cite{CheahThesis}).
\item The natural morphism $\sigma:=(\phi, \rho)\colon A^{[n-1, n]}\to A^{[n-1]}\times A$ is the blow up along the universal subscheme $U_{n-1}\subset A^{[n-1]}\times A$. Define a line bundle $\rL:=\sO_{A^{[n-1, n]}}(-E)$ on $A^{[n-1, n]}$, where $E$ is the exceptional divisor of the blow up. 
\item The natural morphism $\sigma=(\phi, \rho)\colon A^{[n-1, n]}\to A^{[n-1]}\times A$ is also identified with the projection $$\P(\rI_{n-1})=\Proj\left(\Sym \rI_{n-1}\right)\to A^{[n-1]}\times A.$$ Then $\rL$ is identified with  $\sO_{\P(\rI_{n-1})}(1)$ .
\item The morphism $\psi$ is generically finite of degree $n$.
\item The natural morphism $(\psi, \rho) \colon A^{[n-1, n]}\to A^{[n]}\times A$ is identified with the projection $$\P(\omega_{U_{n}})\to A^{[n]}\times A,$$ where $\omega_{U_{n}}$ is the relative dualising sheaf (supported on $U_{n}$) of the universal subscheme $U_{n}\subset A^{[n]}\times A$.
\end{itemize}

Before we return to the proof of Proposition \ref{prop:2}, we do the following calculation:
\begin{lemma}\label{lemma:chern}
Let $A$ be an abelian surface, $\Delta_A\subset A\times A$ be the diagonal. Then in $\CH(A\times A)$, $c_1\left(\sO_{\Delta_A}\right)=c_3\left(\sO_{\Delta_A}\right)=c_4\left(\sO_{\Delta_A}\right)=0$, and    $c_2\left(\sO_{\Delta_A}\right)=-\Delta_A$.
\end{lemma}
\begin{proof}
We apply the Grothendieck-Riemann-Roch formula to the diagonal embedding $A\inj A\times A$, we get (since $\td\left(T_A\right)=\td\left(T_{A\times A}\right)=1$): $\ch(\sO_{\Delta_A})=\Delta_A \in \CH^2(A\times A)$, and the calculation of Chern classes follows.
\end{proof}

\begin{proof}[Proof of Proposition \ref{prop:2}]
We do induction on $n$. When $n=0$, there is nothing to prove. \\
When $n=1$, the only possible $\miu=(1)$, hence $E_{\miu, m}$ is the identity correspondence of $A^{m+1}$. Since $\rO_1$ is the structure sheaf, $T_1=T_A$ is trivial and $\rI_1=\sI_{\Delta_A}$ is the ideal sheaf of the diagonal, whose Chern classes are either zero or $\Delta_A$ (by Lemma \ref{lemma:chern}), Proposition \ref{prop:2} is verified in this case.\\
Now assuming the statement holds for $n-1$, let us prove it for $n$. In the rest of the proof, a \emph{partition} of $n$ means a grouping of the set $\{1, 2, \cdots, n\}$ rather than just a decreasing sequence of natural numbers with sum $n$ as before. More precisely, a partition $\miu$ of length $l$ is a sequence of mutually exclusive subsets $\miu_1, \cdots, \miu_l\in 2^{\{1,\cdots, n\}}$ such that $\coprod_{j=1}^l \miu_j=\{1, \cdots, n\}$. Thus we can naturally identify $A^{\miu}:= A^{l_{\miu}}$ with the diagonal $\left\{(x_1, \dots, x_n)\in A^n~|~x_i=x_j \text{ if $i, j\in\miu_k$ for some $k$}\right\}\subset A^n$.\\
 Consider the reduced fibre product $(A^{\miu}\times_{A^{(n)}}A^{[n-1, n]})_{red}$, which has $l_{\miu}$ irreducible components dominating $A^{\miu}$, depending on the choice of the residue point. Let us pick one component, for example, the one where over a general point $(x_1, \cdots, x_n)\in A^{\miu}$, the residue point is $x_n$. Let $\miu'$ be the partition of $\{1,2, \cdots, n-1\}$ given by $\miu_i':=\miu_i\backslash\{n\}$ for all $i$. Let us call this irreducible component $E_{\miu, \miu'}$. Set theoretically, $$E_{\miu, \miu'}=\left\{((x_1,\cdots, x_n), z'\subset z)\in A^{\miu}\times A^{[n-1, n]}~|~ [z']=x_1+\cdots+x_{n-1}, [z]=x_1+\cdots+x_n \right\};$$
 $$E_{\miu}=\left\{((x_1,\cdots, x_n), z)\in A^{\miu}\times A^{[n]}~|~ [z]=x_1+\cdots+x_n \right\};$$
 $$E_{\miu'}=\left\{((x_1,\cdots, x_{n-1}), z')\in A^{\miu'}\times A^{[n-1]}~|~ [z']=x_1+\cdots+x_{n-1} \right\},$$
  where $[-]$ means the Hilbert-Chow morphism. \\
  We have the following commutative diagram with natural morphisms:
  \begin{equation}\label{diagram:2}
  \xymatrix{
  A^{\miu'}\times A & & A^{\miu} \ar[ll]_{\iota}\\
  E_{\miu'}\times A \ar[u]^{g_{\miu',1}} \ar[d]_{f_{\miu',1}} & E_{\miu, \miu'} \ar[l]_{\chi'} \ar[r]^{\chi} \ar[d]^{p} & E_{\miu} \ar[u]_{g_{\miu}} \ar[d]^{f_{\miu}}\\
  A^{[n-1]}\times A & A^{[n-1, n]}\ar[l]^(0.4){\sigma=(\phi, \rho)} \ar[r]_{\psi} & A^{[n]}
  }
  \end{equation}
Here and in the sequel, for any morphism $h$ and any $m\in\N$, we denote by $h_m$ the morphism $h\times \id_{A^m}$. In the above diagram, $f_{\miu}$, $g_{\miu}$, $f_{\miu'}$, $g_{\miu'}$ are the natural projections; $\chi=(\id_{A^{\miu}}, \psi)\colon \left((x_1,\cdots, x_n), z'\subset z\right)\mapsto ((x_1,\cdots, x_n), z)$, $\chi'=(\pr_{A^{\miu'}}, \sigma)\colon \left((x_1,\cdots, x_n), z'\subset z\right)\mapsto \left((x_1,\cdots, x_{n-1}), z', x_n\right)$ both are of degree 1; and finally $\iota\colon (x_1,\cdots, x_n) \mapsto \left((x_1,\cdots, x_{n-1}),x_n\right)$ is either an isomorphism or a diagonal embedding depending on whether $n$ is the only one element in the subset of partition where $n$ belongs to.

Here comes the key setting for the induction process. For any $m\in \N$, we make a product of the above diagram with $A^m$ and replace any morphism $h$ by $h_m:=h\times \id_{A^m}$: 
\begin{equation}\label{diagram:3}
  \xymatrix{
  A^{\miu'}\times A^{m+1} & & A^{\miu}\times A^m \ar[ll]_{\iota_m}\\
  E_{\miu'}\times A^{m+1} \ar[u]^{g_{\miu',m+1}} \ar[d]_{f_{\miu',m+1}} & E_{\miu, \miu'}\times A^m \ar[l]_{\chi'_m} \ar[r]^{\chi_m} \ar[d]^{p_m} & E_{\miu}\times A^m \ar[u]_{g_{\miu,m}} \ar[d]^{f_{\miu,m}}\\
  A^{[n-1]}\times A^{m+1} & A^{[n-1, n]}\times A^m\ar[l]^{\sigma_m} \ar[r]_{\psi_m} & A^{[n]}\times A^m
  }
  \end{equation}

Given $\gamma\in \CH(A^{[n]}\times A^m)$ a polynomial expression as in Proposition \ref{prop:2}, we want to prove that $g_{\miu,m *}\left(f_{\miu,m}^* \gamma\right)\in \CH(A^{\miu}\times A^m)$ is a polynomial of big diagonals of $A^{l_{\miu}+m}$. Since $\iota_m$ is either an isomorphism or a diagonal embedding, it suffices to prove the same thing for $\iota_{m *}\circ g_{\miu,m *}\left(f_{\miu,m}^* \gamma\right)\in \CH(A^{\miu'}\times A^{m+1})$. However,
\begin{eqnarray*}
\iota_{m *}\circ g_{\miu,m *}\left(f_{\miu,m}^* \gamma\right)&=& \iota_{m *}\circ g_{\miu,m *}\circ \chi_{m *}\circ \chi_{m}^*\circ f_{\miu,m}^*\left( \gamma\right)~~ \text{(since $\chi_m$ is of degree 1)}\\
&=& g_{\miu',m+1 *}\circ\chi'_{m *}\circ \chi_{m}^*\circ f_{\miu,m}^*\left( \gamma\right)\\
&=& g_{\miu',m+1 *}\circ\chi'_{m *}\circ p_{m}^*\circ \psi_{m}^*\left(\gamma\right)\\
&=& g_{\miu',m+1 *}\circ f_{\miu',m+1}^*\circ\sigma_{m *} \circ \psi_{m}^*\left(\gamma\right) ~~~\text{(by \cite[Page 626 (2.13)]{MR2435839})}.\\
\end{eqnarray*}
Using the induction hypothesis (since $\miu'$ is a partition of $n-1$), we find that to finish the proof, it is enough to verify
\begin{prop}\label{prop:recurrence}
If $\gamma\in \CH(A^{[n]}\times A^m)$ is a polynomial expression in the cycles of the following forms:
\begin{itemize}
  \item $\pr_0^*\left(c_j\left(\rO_n\right)\right)$ for some $j\in \N$;
  \item $\pr_0^*\left(c_j\left(T_n\right)\right)$ for some $j\in \N$;
  \item $\pr_{0i}^*\left(c_j\left(\rI_n\right)\right)$ for some $1\leq i\leq m$ and $j\in \N$;
  \item $\pr_{ij}^*\left(\Delta_{A}\right)$ for some $1\leq i\neq j\leq m$.
\end{itemize}
then $\sigma_{m *} \circ \psi_{m}^*\left(\gamma\right)\in \CH(A^{[n-1]}\times A^{m+1})$ is a polynomial in cycles of these four forms with $n$ replaced by $n-1$ and $m$ replaced by $m+1$.
\end{prop}

This will follow essentially from the formulae below established in \cite{MR1795551}. We adopt the notation in Diagram (\ref{diagram:1}), (\ref{diagram:2}), (\ref{diagram:3}) and the definition of the line bundle $\rL$ after Diagram (\ref{diagram:1}). In our case of abelian surface, the formulae are simplified thanks to the fact that $T_A$ is trivial. 

\begin{thm}[\cite{MR1795551}, Proposition 2.3, Lemma 2.1, in the proof of Proposition 3.1 and Lemma 1.1]\label{thm:EGL}
We have the following equalities in the Grothendieck group $K_0(A^{[n-1, n]})$:
\begin{itemize}
 \item[$(i)$] $\psi^! T_n=\phi^!T_{n-1}+\rL\cdot \sigma^!\rI\dual_{n-1}+1$;
 \item[$(ii)$]  $\psi^!\rO_n=\phi^!\rO_{n-1}+\rL$;
\end{itemize}
an equality in the Grothendieck group $K_0(A^{[n-1, n]}\times A)$:
\begin{itemize}
 \item[$(iii)$] $\psi_1^!\rI_n=\phi_1^!\rI_{n-1}-(\rL\boxtimes \sO_A)\cdot \rho_1^!(\sO_{\Delta_A})$;
\end{itemize}
and an equality in the Chow group $\CH(A^{[n-1]}\times A)$:
\begin{itemize}
 \item[$(iv)$] $\sigma_*\left(c_1(\rL)^i\right)=(-1)^ic_i(-\rI_{n-1})$.
\end{itemize}
\end{thm}

Return to the proof of Proposition \ref{prop:recurrence}. Taking the Chern classes of both sides of $(i), (ii), (iii)$ in Theorem \ref{thm:EGL}, we get formulae for pull-backs by $\psi$ or $\psi_1$ of the Chern classes of $T_n, \rO_n, \rI_n$ in terms of polynomial expressions of the first Chern class of $\rL$ and the pull-backs by $\phi, \rho, \sigma$ of the Chern classes of $T_{n-1}, \rO_{n-1}, \rI_{n-1}$ and $\sO_{\Delta_A}$. Therefore by the calculations in Lemma \ref{lemma:chern} and the fact that $\sigma=(\phi, \rho)$, we obtain that $\psi_m^*(\gamma)\in \CH\left(A^{[n-1, n]}\times A^m\right)$  is a polynomial of cycles of the following five forms:
\begin{itemize}
\item $\sigma^*_{m}\circ \pr_0^* \left(c_j(T_{n-1})\right)$ for some $j\in \N$;
\item $\pr_0^*\left(c_1(\rL)\right)$;
\item $\sigma^*_m\circ\pr^*_{0i}\left(c_j(\rI_{n-1})\right)$ for some $1\leq i\leq m+1$ and $j\in\N$;
\item $\sigma^*_{m}\circ \pr_0^* \left(c_j(\rO_{n-1})\right)$ for some $j\in \N$;
\item $\sigma^*_m\circ\pr_{ij}^*\left(\Delta_A\right)$ for some $1\leq i\neq j\leq m+1$,
\end{itemize}
where we also use $\pr_0$ to denote the projection $A^{[n-1, n]}\times A^m\to A^{[n-1, n]}$, \etc\\
When apply $\sigma_{m*}$ to a polynomial in cycles of the above five types, using the projection formula for the birational morphism $\sigma_m$ and Theorem \ref{thm:EGL}$(iv)$, we conclude that $\sigma_{m *} \circ \psi_{m}^*\left(\gamma\right)$ is of the desired form. This finishes the proof of Proposition \ref{prop:recurrence} thus completes the proof of Proposition \ref{prop:Voisinhard}.
\end{proof}


\section*{Acknowledgements}
I would like to express my gratitude to Claire Voisin for her excellent mini-course at the colloquium GRIFGA as well as for the helpful discussions afterwards, to Kieran O'Grady for raising the question at the colloquium which is the main subject of this paper. I would like to thank Giuseppe Ancona for explaining me the result of O'Sullivan, to thank Ulrike Greiner for pointing out to me a gap in a first version and helping me to fix it. Finally, I want to thank Zhi Jiang for his careful reading of the preliminary version of this paper and the referees for their helpful suggestions which improved the paper a lot.

\bibliographystyle{amsplain}
\bibliography{biblio_fulie}

\end{document}